\documentclass[11pt,reqno]{amsart}
\pdfoutput=1 
\usepackage{amssymb,latexsym}
\usepackage{graphicx}
\usepackage{url}		

\calclayout

\newcommand{\Z}{{\mathbb Z}}

\theoremstyle{plain}
\numberwithin{equation}{section}
\newtheorem{thm}{Theorem}[section]
\newtheorem{theorem}[thm]{Theorem}
\newtheorem{lemma}[thm]{Lemma}
\newtheorem{corollary}[thm]{Corollary}

\newtheorem{example}[thm]{Example}
\newtheorem{definition}[thm]{Definition}

\newtheorem{remark}[thm]{Remark}

\renewcommand{\labelenumi}{\theenumi}
\renewcommand{\theenumi}{(\alph{enumi})}
\allowdisplaybreaks
\newcommand{\Fix}{{\mathrm{Fix}}}

\usepackage{array}
\newcolumntype{R}{>{\raggedleft\arraybackslash$}r<{$}}
\newcolumntype{C}{>{\centering\arraybackslash$}c<{$}}
\newcolumntype{L}{>{\raggedright\arraybackslash$}l<{$}}
\usepackage{booktabs}

\usepackage{hyperref}
\begin{document}

\title{Fibonacci Cycles and Fixed Points}
\author{Walter A. Kehowski}
\address{Department of Mathematics and Computer Science\\
                Glendale Community College\\
                Glendale, AZ 85308}
\email{walter.kehowski@gccaz.edu}
\thanks{}

\begin{abstract}
Let $S_b(n)$ denote the sum of the squares of the digits of the positive integer $n$ in base $b\geq2$. It is well-known that the sequence of iterates of $S_b(n)$ terminates in a fixed point or enters a cycle. Let $N=2n-1$, $n\geq2$. It is shown that if $b=F_{N+1}$, then a cycle of $S_b$ exists with initial term $F_{N}=F_{0}.F_{N}$, and terminal element $F_{n}.F_{n-1}$ if $n$ is even, or terminal element $F_{n-1}.F_{n}$ if $n$ is odd. Similarly, Let $N=2n+1$, $n\geq1$. If $b=F_{N-1}$, then a cycle of $S_b$ exists with initial term $F_{N}=F_{2}.F_{N-2}$, and terminal element $F_{n}.F_{n+1}$ if $n$ is even, or terminal element $F_{n+1}.F_{n}$ if $n$ is odd. Furthermore, the cycles also admit extension as an arithmetic sequence of cycles of $S_b$ with base $b=F_{N+1}+F_{N+2}k$ and $b=F_{N-1}+F_{N-2}k$, respectively. Some fixed points of $S_b$ with $b$ a Fibonacci base are shown to exist. Lastly, both cycles and fixed points admit further generalization to Pell polynomials.
\end{abstract}

\maketitle

\section{Introduction}\label{s:intro}

Let $b\geq2$ be any number base and let $S_b(n)$ denote the sum of squares of the digits of the positive integer $n$ in base $b$. It is known that the iterates of $S_b$ on any positive integer eventually enter a cycle or terminate in a fixed point. It is the purpose  of this paper to demonstrate the existence of certain cycles and fixed points of $S_b$, where $b=F_{2n}$, and $F_n$ refers to the Fibonacci sequence: $F_0=0$, $F_1=1$, $F_n=F_{n-1}+F_{n-2}$, $n\geq2$. See Section \ref{s:fibonacci}. Digits in base $b$ are separated by periods, for example, $x.y|_b=xb+y$. The $|_b$ will be omitted if the base $b$ is understood. We will say ``$x$ is fixed in base $b$'' to mean that $x$ is a fixed point for $S_b$, that is, $S_b(x)=x$. Once an element of a cycle is designated as an initial element $x$, the terminal element of the cycle will that element $z$ such that $S_b(z)=x$.

\subsection{Fundamental cycles}

\begin{example}[Fundamental cycle of type I]
A cycle under $S_b$ starting with $F_{N}=0.F_{N}$ and $b=F_{N+1}$, $N=2n-1$, $n\geq2$, is called a \emph{fundamental cycle of type I}. The elements of the cycle follow in general from identity \eqref{e:fibplus}, namely,
$$
 F_i^2+F_{N-i}^2 = F_{N-(2j+1)}F_{N+1} + F_{2j+1},
$$
where $j=\min(i,N-i)$ and $0 \leq i \leq n-1$. Furthermore, the cycle ends with $F_{n}.F_{n-1}$ if $n$ is even, and ends with $F_{n-1}.F_{n}$ if $n$ is odd. See Section \ref{s:fibonacci-i}.  For example, suppose the initial term is $F_{11}=89$, with $b=F_{12}=144$. See Table \ref{tab:fibcycle-example-type-I}. Observe that the indices of each element sum to $N=11$. The terminal element $F_{6}.F_{5}$ of the cycle follows from Lucas's identity \eqref{eqn:lucas}.
\begin{center}
\setlength{\extrarowheight}{4pt}
\begin{table}
\begin{tabular}{*{19}{@{\hspace{2pt}}R@{\hspace{2pt}}}}\toprule
 & & & & & & F_{11} &=& 89 &=& 0\cdot144 &+& 89 &=& F_{0}F_{12} &+& F_{11} &=& F_{0}.F_{11} \\
 F_0^2 &+& F_{11}^2 &=& 0^2 &+& 89^2 &=& 7921 &=& 55\cdot144 &+& 1 &=& F_{10}F_{12} &+& F_{1} &=& F_{10}.F_{1} \\
 F_{10}^2 &+& F_{1}^2 &=& 55^2 &+& 1^2 &=& 3026 &=& 21\cdot144 &+& 2 &=& F_{8}F_{12} &+& F_{3} &=& F_{8}.F_{3} \\
 F_{8}^2 &+& F_{3}^2 &=& 21^2 &+& 2^2 &=& 445 &=& 3\cdot144 &+& 13 &=& F_{4}F_{12} &+& F_{7} &=& F_{4}.F_{7} \\
 F_{4}^2 &+& F_{7}^2 &=& 3^2 &+& 13^2 &=& 178 &=& 1\cdot144 &+& 34 &=& F_{2}F_{12} &+& F_{9} &=& F_{2}.F_{9} \\
 F_{2}^2 &+& F_{9}^2 &=& 1^2 &+& 34^2 &=& 1157 &=& 8\cdot144 &+& 5 &=& F_{6}F_{12} &+& F_{5} &=& F_{6}.F_{5} \\
 F_{6}^2 &+& F_{5}^2 &=& 8^2 &+& 5^2 &=& 89 &=& \multicolumn{9}{L}{F_{11}.}\\\bottomrule
\end{tabular}
\caption{\label{tab:fibcycle-example-type-I} Computation of the fundamental cycle of type I with initial element $F_{11}=0.F_{11}|_b$ and base $b=F_{12}$.}
\end{table}
\end{center}
\end{example}

\begin{example}[Fundamental cycle of type II]
A cycle under $S_b$ starting with $F_{N}$ and $b=F_{N-1}$, $N=2n+1$, $n\geq2$, is called a \emph{fundamental cycle of type II}. The elements of the cycle follow in general from identity \eqref{e:fibminus}, namely,
$$
 F_i^2+F_{N-i}^2 = F_{N-(2j-1)}F_{N-1} + F_{2j-1},
$$
where $j=\min(i,N-i)$ and $1 \leq i \leq n+1$. Furthermore, the cycle ends with $F_{n}.F_{n+1}|_b$ if $n$ is even, and ends with $F_{n+1}.F_{n}|_b$ if $n$ is odd. See Section \ref{s:fibonacci-ii}. For example, suppose the initial term is $F_{13}=233$, with $b=F_{12}=144$. A computation similar to that of Table \ref{tab:fibcycle-example-type-I} shows that the cycle is
$$
 F_{13}=F_{2}.F_{11}|_b,\ F_{10}.F_{3}|_b,\ F_{8}.F_{5}|_b,\ F_{4}.F_{9}|_b,\ F_{6}.F_{7}|_b.
$$
Observe that the indices of each term sum to $N=13$. The terminal element $F_{6}.F_{7}|_b$ of the cycle follows from Lucas's identity \eqref{eqn:lucas}.
\end{example}

\subsection{Fixed points}\label{intro-fixed}

Fixed points occur in base $b=F_{6n-2}$, with fixed point $F_{2n}.F_{4n-1}|_b$, and  in base $b=F_{6n+2}$, with fixed point $F_{2n}.F_{4n+1}|_b$. They are isolated fixed points, that is, they have no preimage under their respective $S_b$. See Section \ref{s:fixedpoints}.

\subsection{Arithmetic sequences of cycles}\label{intro-arith}

The cycles of type I admit an extension to an arithmetic sequence of cycles. Namely, if $b=F_{N+1}+F_{N+2}k$, $N=2n-1$, $n\geq2$, there exists a cycle with initial term $F_{0}+F_{1}k.F_{N}+F_{N+1}k|_b$, and terminal element $F_{n}+F_{n+1}k.F_{n+1}+F_{n+2}k|_b$ if $n$ is even, and terminal element $F_{n+1}+F_{n+2}k.F_{n}+F_{n+1}k|_b$ if $n$ is odd. See Section \ref{s:arithFibonacci-i}.

The cycles of type II also admit an extension to an arithmetic sequence of cycles. Namely, if $b=F_{N-1}+F_{N-2}k$, $N=2n+1$, $n\geq2$, there exists a cycle with initial term $F_{0}+F_{1}k.F_{N}+F_{N+1}k|_b$, and terminal element $F_{n}+F_{n+1}k.F_{n+1}+F_{n+2}k|_b$ if $n$ is even, and terminal element $F_{n+1}+F_{n+2}k.F_{n}+F_{n+1}k|_b$ if $n$ is odd. See Section \ref{s:arithFibonacci-ii}.

\subsection{Pell cycles and fixed points}

The results mentioned in Subsections \ref{intro-fixed} and \ref{intro-arith} can be generalized to Pell polynomials. See Section \ref{s:pell}.

\begin{theorem}[Theorem \ref{thm:pell-fixed-points}]
 Let $n$ be a positive integer.
\begin{enumerate}
\renewcommand{\labelenumi}{\theenumi}
\renewcommand{\theenumi}{(\alph{enumi})}

  \item The polynomial $p_{2n}(x).p_{4n-1}(x)|_b$ is a fixed point of $S_b$, where $b=p_{6n-2}(x)$.
  \item The polynomial $p_{2n}(x).p_{4n+1}(x)|_b$ is a fixed point of $S_b$, where $b=p_{6n+2}(x)$.
  \item The polynomial $p_{2n}(x)p_{2n-1}(x).p_{2n+1}(x)p_{2n-1}(x)|_b$ is a fixed point of $S_b$, where $b=p_{4n}(x)$.

\end{enumerate}
\end{theorem}

\begin{corollary}[Corollary \ref{cor:arithpellfixed}]
 Assume $n\geq1$ and $k\geq0$. The polynomial $p_{2n}(x)u.p_{2n+1}(x)u|_{b}$ is a fixed point in base $b=p_{4n}(x)+p_{2n+1}(x)p_{4n+1}(x)k$, where $u=p_{2n-1}(x)+p_{2n}(x)p_{2n+1}(x)k$.
\end{corollary}

\begin{theorem}[Theorem \ref{thm:pellplus}]
 The iterates of $S_b$, $b=p_{2n}(x)$, $n\geq2$, on $p_0(x).p_{2n-1}(x)|_b$ comprise a cycle with initial element $p_0(x).p_{2n-1}(x)|_b$ and terminal element $p_{n}(x).p_{n-1}(x)|_b$ if $n$ is even, or terminal element $p_{n-1}(x).p_{n}(x)|_b$ if $n$ is odd.
\end{theorem}

\begin{theorem}[Theorem \ref{thm:arithpellplus}]
 The iterates of $S_b$, $b=p_{2n}(x)+p_{2n+1}(x)k$, on $p_0(x)+p_{1}(x)k.p_{2n-1}(x)+p_{2n}(x)k|_b$ via \eqref{e:pell-type-II} comprise a cycle with initial element
 $$
 p_0(x)+p_{1}(x)k.p_{2n-1}(x)+p_{2n}(x)k
 $$
 and terminal element
 $$
 p_{n}(x)+p_{n+1}(x)k.p_{n+1}(x)+p_{n+2}(x)k
 $$
 if $n$ is even, or terminal element
 $$
 p_{n+1}(x)+p_{n+2}(x)k.p_{n}(x)+p_{n+1}(x)k
 $$
 if $n$ is odd.
\end{theorem}

\begin{theorem}[Theorem \ref{thm:pellminus}]
 The iterates of $S_b$, $b=p_{2n}(x)$, on $p_2(x).p_{2n-1}(x)|_b$ via \eqref{e:pell-type-I} comprise a cycle with initial element $p_2(x).p_{2n-1}(x)|_b$ and terminal element $p_{n}(x).p_{n+1}(x)|_b$ if $n$ is even, or terminal element $p_{n+1}(x).p_{n}(x)|_b$ if $n$ is odd.
\end{theorem}

\begin{theorem}[Theorem  \ref{thm:arithpellminus}]
 Let $N=2n+1$, $n\geq 1$, and $k\geq0$. Then the iterates of $S_b$, $b=p_{N-1}(x)+p_{N-2}(x)k$, yield via \eqref{e:pellfibcycle-ii} a cycle with initial element
 $$
 p_{2}(x)+p_{1}(x)k.p_{N-2}(x)+p_{N-3}(x)k|_b
 $$
 and terminal element
 $$
 p_{n}(x)+p_{n-1}(x)k.p_{n-1}(x)+p_{n-2}(x)k|_b
 $$
 if $n$ is even, or terminal element
 $$
 p_{n-1}(x)+p_{n-2}(x)k.p_{n}(x)+p_{n-1}(x)k|_b
 $$
 if $n$ is odd.
\end{theorem}

\section{A summary of Beardon}\label{s:beardon}

The following results are from Beardon, \cite{beardon}.

\begin{lemma}[\cite{beardon}]
 Suppose $n$ has at least four digits in base $b$. Then $S_b(n)$ has fewer digits than $n$.
\end{lemma}

\begin{lemma}[\cite{beardon}]
 If $n$ has at most three digits in base $b$, then so does $S_b(n)$.
\end{lemma}

\begin{theorem}
 For any positive integer $n$, the successive images of $S_b$ either terminate in a fixed point or enter a cycle.
\end{theorem}

If $n$ is a fixed point of $S_b$, then $n$ is \emph{nontrivial} if $n>1$ and \emph{trivial} if $n=1$. Since a fixed point can be regarded as a cycle of length one, a cycle will always be assumed to have length at least two.

\begin{theorem}[\cite{beardon}]
 Any nontrivial fixed point of $S_b$ has exactly two digits in base $b$.
\end{theorem}

\begin{theorem}[\cite{beardon}]\label{thm:fixedpoints}
 The number $x.y|_b$ is a fixed point of $S_b$ if and only if $(x,y)$ is a solution to the equation
 \begin{equation}\label{eqn:fixedpoints}
  (2x-b)^2 + (2y-1)^2 = 1+b^2,
 \end{equation}
where $0 \leq x <b$, $1 \leq y <b$.
\end{theorem}

\begin{proof}
 Complete the square on $x^2+y^2=xb+y$.
\end{proof}

\begin{thm}[\cite{beardon}]\label{thm:numberoffixedpoints}
 Let $\Fix_b$ be the set of fixed points of $S_b$. Then $|\Fix_b| = d(1+b^2)-1$, where $d(1+b^2)$ is the number of divisors of $1+b^2$.
\end{thm}

\begin{corollary}[\cite{beardon}]
 For each base $b$, $S_b$ has only the trivial fixed point if and only if $1+b^2$ is prime.
\end{corollary}

Let $(u,v)$ be a solution to $u^2+v^2=1+b^2$, where $u$ has the same parity as $b$, $-b \leq u < b$, and $v$ is odd, $0<v<b$. The fixed points $x.y_b$ of $S_b$ are then given by all $(x,y)$ such that
\begin{equation}\label{eqn:fixedpointsuv}
 2x-b= u, \quad 2y-1=v.
\end{equation}

\begin{example}[Fixed points in base $b=12$]\label{exm:fixedpoints}
 If $b=12$, then all solutions to $u^2+v^2=1+12^2=145$ relevant to \eqref{eqn:fixedpointsuv} are $(-12,1)$, $(-8,9)$, and $(8,9)$. The first possibility $(-b,1)$ always gives the trivial fixed point. The other two possibilities are easily mentally computed to be $(x,y)=(2,5)$ and $(x,y)=(10,5)$, so that $\Fix_{12}=\{1,2.5,10.5\}_{12}=\{1,29,125\}$. By Theorem \ref{thm:numberoffixedpoints}, the number of fixed points is $d(5\cdot 29)-1=3$.
\end{example}

\begin{theorem}[\cite{beardon}]
 Any cycle of $S_b$ is a subset of $\{1, \dots, 2b^2-1\}$.
\end{theorem}

Let $C$ be a cycle of $S_b$, and let an element $a$ of $C$ be designated as the \emph{initial element}. The \emph{terminal element} of $C$ is then that element $z$ of the cycle such that $S_b(z)=a$.  The initial element is chosen for convenience but will often be the smallest element of $C$. If the initial element of a cycle is also the smallest element, then the cycle is said to be in \emph{standard form}. The fundamental cycles in Sections \ref{s:fibonacci-i} and \ref{s:fibonacci-ii} are in standard form. We define $\Z|_b =\{1,\dots b^2-1\}$, $b>2$, since all elements of a cycle in this paper will have at most two digits.

\begin{remark}
 Let $\mathcal{C}_b$ be the set of all cycles of $S_b$. There is no known formula to compute $|\mathcal{C}_b|$ or algorithm besides direct search to determine $\mathcal{C}_b$.
\end{remark}

\begin{example}[Cycles in base $b=12$]
 The set of cycles $\mathcal{C}_{12} $ consists of the following:
 \begin{enumerate}
 \item $\{ 5, 2.1 \}_{12}=\{5, 25\}$;
 \item $\{8, 5.4, 3.5, 2.10, 8.8, 10.8, 1.1.8, 5.6, 5.1, 2.2\}_{12}=\{8, 64, 41, 34, 104, 128, 164, 66, 61, 26\}$;
 \item $\{1.8, 5.5, 4.2\}_{12}=\{20, 65, 50\}$;
 \item $\{6.8, 8.4\}_{12}=\{80, 100\}$.
 \end{enumerate}
\end{example}

\section{Fibonacci numbers and identities}\label{s:fibonacci}

The \emph{Fibonacci sequence} $\left(F_{n}\right)_{n=0}^{\infty}$ is recursively defined by $F_0=0$, $F_1=1$, $F_{n}=F_{n-1}+F_{n-2}$, $n\geq2$, \cite{A000045}. The following identities will be used frequently. The book \cite{koshy2001} by T.~Koshy is a comprehensive resource. See Section \ref{s:pell} for generalization of the identities to the Pell polynomials.

\begin{description}
  \item[Cassini's identity]
  \begin{equation}\label{eqn:cassini}
    F_{n}^2 = F_{n-1}F_{n+1} - (-1)^n
  \end{equation}
  \item[Catalan's identity]
  \begin{equation}\label{eqn:catalan}
    F_{n}^2 = F_{n+r}F_{n-r} + (-1)^{n-r} F_r^2
  \end{equation}
  \item[Vajda's identity]
  \begin{equation}\label{eqn:vajda}
    F_{n+r} F_{n+s} = F_{n}F_{n+r+s} + (-1)^n F_r F_s
  \end{equation}
  \item[Lucas's identity]
  \begin{equation}\label{eqn:lucas}
    F_{2n+1} = F_{n+1}^2 + F_{n}^2
  \end{equation}
  \item[d'Ocagne's identity]
  \begin{equation}\label{eqn:docagne}
    F_{2n} = F_{n+1}^2 - F_{n-1}^2
  \end{equation}
\end{description}

\section{Fibonacci cycles of type I}\label{s:fibonacci-i}

The author observed that cycles in base $b=F_{2n}$, $n\geq2$, \cite{A001906}, and initial term $F_{2n-1}=F_{0}.F_{2n-1}|_b$, \cite{A001519},  all have digits in the Fibonacci sequence. See Table \ref{tab:fibcycle-type-I} for some examples. We choose $F_1=1$ so that the sum of the indices of the second term of each cycle is $N=2n-1$. They are a consequence of the following generalization of Lucas's identity \eqref{eqn:lucas}.

\begin{center}
\begin{table}
\begin{tabular}{c*{9}{r@{.}l}}\toprule
 Base $b$ & \multicolumn{18}{l}{Fundamental cycle} \\\midrule
$F_{4}$ & $F_{0}$&$F_{3}$, &  $F_{2}$&$F_{1}$ \\
$F_{6}$ & $F_{0}$&$F_{5}$, &  $F_{4}$&$F_{1}$, & $F_{2}$&$F_{3}$\\
$F_{8}$ & $F_{0}$&$F_{7}$, &  $F_{6}$&$F_{1}$, & $F_{4}$&$F_{3}$\\
$F_{10}$ & $F_{0}$&$F_{9}$, & $F_{8}$&$F_{1}$, & $F_{6}$&$F_{3}$, & $F_{2}$&$F_{7}$, & $F_{4}$&$F_{5}$\\
$F_{12}$ & $F_{0}$&$F_{11}$, & $F_{10}$&$F_{1}$, & $F_{8}$&$F_{3}$, & $F_{4}$&$F_{7}$, & $F_{2}$&$F_{9}$, & $F_{6}$&$F_{5}$\\
$F_{14}$ & $F_{0}$&$F_{13}$, & $F_{12}$&$F_{1}$, & $F_{10}$&$F_{3}$, & $F_{6}$&$F_{7}$\\
$F_{16}$ & $F_{0}$&$F_{15}$, & $F_{14}$&$F_{1}$, & $F_{12}$&$F_{3}$, & $F_{8}$&$F_{7}$\\
$F_{18}$ & $F_{0}$&$F_{17}$, & $F_{16}$&$F_{1}$, & $F_{14}$&$F_{3}$, & $F_{10}$&$F_{7}$, & $F_{2}$&$F_{15}$, & $F_{12}$&$F_{5}$, & $F_{6}$&$F_{11}$, & $F_{4}$&$F_{13}$, & $F_{8}$&$F_{9}$\\
$F_{20}$ & $F_{0}$&$F_{19}$, & $F_{18}$&$F_{1}$, & $F_{16}$&$F_{3}$, & $F_{12}$&$F_{7}$, & $F_{4}$&$F_{15}$, & $F_{10}$&$F_{9}$
\\\bottomrule
\end{tabular}
\caption{\label{tab:fibcycle-type-I} Some Fibonacci cycles of type I in bases $b=F_{2n}$ with initial terms $F_{2n-1}=F_{0}.F_{2n-1}|_b$. See also Table \ref{tab:pluscycles}.}
\end{table}
\end{center}

\begin{theorem}\label{thm:fibplus}
 Let $N=2n-1$, $n\geq1$. Then
\begin{equation}\label{e:fibplus}
 F_{N-i}^2 + F_i^2 = F_{N-(2j+1)} F_{N+1} + F_{2j+1},
\end{equation}
where $j=\min(i,N-i)$ and $0\leq i \leq n-1$. Observe that $N-(2j+1)$ is always even, and that the indices $N-(2j+1)$ and $2j+1$ sum to $N$. If $i=n-1$, we have Lucas's identity \eqref{eqn:lucas}.
\end{theorem}

\begin{proof}
Let us assume that $j$ is the smaller of $i$ and $N-i$, $0\leq i \leq n-1$. By Catalan's identity \eqref{eqn:catalan} and Lucas's identity \eqref{eqn:lucas},
\begin{align*}
 F_{N-i}^2 + F_i^2 &= F_{N-j}^2 + F_j^2 \\
 &= F_{(N+1)-2(j+1)} F_{(N+1)} + (-1)^{(N+1)-2(j+1)}F_{j+1}^2 + F_j^2 \\
 &= F_{N-(2j+1)} F_{N+1} + F_{j+1}^2 + F_j^2 \\
 F_{N-i}^2 + F_i^2 &= F_{N-(2j+1)} F_{N+1} + F_{2j+1},\quad (0\leq i \leq n-1).\qedhere
\end{align*}
\end{proof}

Let ${\mathcal{P}}_{+}(N)$ be the set of all pairs $[r,s]$ of nonnegative integers, where $r+s=N$ and $s$ is odd, so that $r$ is necessarily even. Note that $|{\mathcal{P}}_{+}(2n-1) |=n$. Define $\psi_{+} : {\mathcal{P}}_{+}(N) \rightarrow {\mathcal{P}}_{+}(N)$ by
$$
\psi_{+}([r,s])=[N-(2t+1),2t+1],\quad t=\min(r,s).
$$
We write $\psi_{+}$ instead of $\psi_{+}^N$ since $N$ is fixed, once chosen. Note that $\psi_{+}([n,n-1])=[0,N]$ if $n$ is even, or $\psi_{+}([n-1,n])=[0,N]$ if $n$ is odd. Now assume $n\geq2$. Define the map $\Psi_{+} : {\mathcal{P}}_{+}(N) \rightarrow \mathbb{Z}|_{b}$ by $\Psi_{+}([r,s])=F_{r}.F_{s}|_b$, where $b=F_{N+1}$. Thus, we have
\begin{equation}\label{eqn:psiplus}
 S_{b} (F_{r}.F_{s}) = \Psi_{+}(\psi_{+}([r,s])).
\end{equation}
Thus, we need only determine the cycles and fixed points of $\psi_{+}$ on ${\mathcal{P}}_{+}(N)$ to obtain the cycles and fixed points of $S_{b}$, $b=F_{N+1}$, on $\mathbb{Z}|_{b}$.

\begin{theorem}[$\psi_{+}$ orbit]\label{thm:psiplus}
 Let $N=2n-1$, $n\geq1$. If $n=1$, then $[0,1]$ is a fixed point of $\psi_{+}$. If $n\geq2$, then the iterates of $\psi_{+}$ on $[0,N]$ comprise a cycle with initial element $[0,N]$ and terminal element $[n,n-1]$ if n is even, or $[n-1,n]$ if $n$ is odd.
\end{theorem}

\begin{proof}
 Observe that $(\psi_{+})^{-1}([r,2s+1])=[N-s,s]$ if $s$ is odd, and $[s,N-s]$ if $s$ is even, where $r=N-(2s+1)$. Further, since $\psi_{+}([n,n-1])=[0,2n-1]$ if $n$ is even, and $\psi_{+}([n-1,n])=[0,2n-1]$ if $n$ is odd, we have that  $[0,N]$ and $[n,n+1]$, $n$ even, or $[n+1,n]$, $n$ odd, are in the same $\psi_{+}$-orbit.
\end{proof}

Let us call the $\psi_{+}$-cycle generated by $[0,N]$ the \emph{fundamental cycle} of $\psi_{+}$, since it occurs for every integer $n\geq2$. Furthermore, we require $n\geq2$ if $F_{2n}$ is to be a meaningful base. See Table \ref{tab:pluscycles} for some cycles and fixed points. See Table \ref{tab:fibcycle-type-I} for some Fibonacci cycles. See also Section \ref{s:fixedpoints} on fixed points.

\begin{center}
\begin{table}
\begin{tabular}{rl}\toprule
 $N$ & Cycles and fixed points \\\midrule
$1$ & [[0,1]] \\
$3$ & [[0, 3], [2, 1]] \\
$5$ & [[0, 5], [4, 1], [2, 3]] \\
$7$ & [[0, 7], [6, 1], [4, 3]] \\
   & [[2, 5]] \\
$9$ & [[0, 9], [8, 1], [6, 3], [2, 7], [4, 5]] \\
$11$ & [[0, 11], [10, 1], [8, 3], [4, 7], [2, 9], [6, 5]] \\
$13$ & [[0, 13], [12, 1], [10, 3], [6, 7]] \\
   & [[2, 11], [8, 5]] \\
   & [[4, 9]] \\
$15$ & [[0, 15], [14, 1], [12, 3], [8, 7]] \\
   & [[2, 13], [10, 5], [4, 11], [6, 9]] \\
$17$ & [[0, 17], [16, 1], [14, 3], [10, 7], [2, 15], [12, 5], [6, 11], [4, 13], [8, 9]] \\
$19$ & [[0, 19], [18, 1], [16, 3], [12, 7], [4, 15], [10, 9]] \\
   & [[2, 17], [14, 5], [8, 11]] \\
   & [[6, 13]] \\
$21$ & [[0, 21], [20, 1], [18, 3], [14, 7], [6, 15], [8, 13], [4, 17], [12, 9], [2, 19], [16, 5], [10, 11]] \\
$23$ & [[0, 23], [22, 1], [20, 3], [16, 7], [8, 15], [6, 17], [10, 13], [2, 21], [18, 5], [12, 11]]\\
   & [[4, 19], [14, 9]] \\\bottomrule
\end{tabular}
\caption{\label{tab:pluscycles} The cycles and fixed points of type I are obtained by the iterates of $\psi_{+}$ on $\mathcal{P}_{+}(N)$, $N=2n-1$, $n\geq1$.}
\end{table}
\end{center}

\begin{corollary}
  Every element of ${\mathcal{P}}_{+}(N)$, $N=2n-1$, $n\geq1$, is either a fixed point or in a cycle of $\psi_{+}$.
\end{corollary}

\begin{proof}
 Let $z$ be an element of ${\mathcal{P}}_{+}(N)$. If $z$ is a fixed point, then we are done. Suppose $z$ is not a fixed point and consider the iterates $\{ \psi_{+}^{(k)}(z) \}$ of $z$.
\end{proof}

The following two theorems are immediate.

\begin{theorem}\label{thm:fullfibplus}
 All fixed points and cycles of $\psi_{+}$ on ${\mathcal{P}}_{+}(N)$, $N=2n-1$, $n\geq2$, generate via \eqref{e:fibplus} Fibonacci fixed points and Fibonacci cycles.
\end{theorem}

\begin{theorem}[Fibonacci cycle of type I]\label{thm:fibpluscycle}
 Let $N=2n-1$, $n\geq2$. The iterates of $S_b$, $b=F_{N+1}$, on $F_0.F_{N}|_b$, comprise via \eqref{e:fibplus} a cycle with initial element $F_0.F_{N}|_b$, and terminal element $F_{n}.F_{n-1}|_b$ if $n$ is even, or terminal element $F_{n-1}.F_{n}|_b$ if $n$ is odd.
\end{theorem}

\section{Fibonacci cycles of type II}\label{s:fibonacci-ii}

The author observed that cycles in base $b=F_{2n}$, $n\geq2$, and initial term $F_{2n+1}=F_{2}.F_{2n-1}|_b$  all have digits in the Fibonacci sequence. See Table \ref{tab:fibcycle-type-II} for some examples. We choose $F_2=1$ so that the sum of the indices of the first term of each cycle is $N=2n+1$. They are a consequence of the following generalization of Lucas's identity \eqref{eqn:lucas}.

\begin{center}
\begin{table}
\begin{tabular}{c*{9}{r@{.}l}}\toprule
Base $b$ & \multicolumn{18}{l}{Fundamental cycle} \\\midrule
$F_{4}$ & $F_{2}$&$F_{3}$\\
$F_{6}$ & $F_{2}$&$F_{5}$, & $F_{4}$&$F_{3}$\\
$F_{8}$ & $F_{2}$&$F_{7}$, &  $F_{6}$&$F_{3}$, &  $F_{4}$&$F_{5}$\\
$F_{10}$ & $F_{2}$&$F_{9}$, &  $F_{8}$&$F_{3}$, &  $F_{6}$&$F_{5}$\\
$F_{12}$ & $F_{2}$&$F_{11}$, &  $F_{10}$&$F_{3}$, &  $F_{8}$&$F_{5}$, &  $F_{4}$&$F_{9}$, &  $F_{6}$&$F_{7}$\\
$F_{14}$ & $F_{2}$&$F_{13}$, &  $F_{12}$&$F_{3}$, &  $F_{10}$&$F_{5}$, &  $F_{6}$&$F_{9}$, &  $F_{4}$&$F_{11}$, &  $F_{8}$&$F_{7}$\\
$F_{16}$ & $F_{2}$&$F_{15}$, &  $F_{14}$&$F_{3}$, &  $F_{12}$&$F_{5}$, &  $F_{8}$&$F_{9}$\\
$F_{18}$ & $F_{2}$&$F_{17}$, &  $F_{16}$&$F_{3}$, &  $F_{14}$&$F_{5}$, &  $F_{10}$&$F_{9}$\\
$F_{20}$ & $F_{2}$&$F_{19}$, &  $F_{18}$&$F_{3}$, &  $F_{16}$&$F_{5}$, &  $F_{12}$&$F_{9}$, &  $F_{4}$&$F_{17}$, &  $F_{14}$&$F_{7}$, &  $F_{8}$&$F_{13}$, &  $F_{6}$&$F_{15}$, & $F_{10}$&$F_{11}$ \\
$F_{22}$ & $F_{2}$&$F_{21}$, &  $F_{20}$&$F_{3}$, &  $F_{18}$&$F_{5}$, &  $F_{14}$&$F_{9}$, &  $F_{6}$&$F_{17}$, &  $F_{12}$&$F_{11}$\\
\bottomrule
\end{tabular}
\caption{\label{tab:fibcycle-type-II} Some fundamental Fibonacci cycles of type II in bases $b=F_{2n}$ with initial terms $F_{2n+1}=F_{2}.F_{2n-1}|_b$. See also Table \ref{tab:minuscycles}.}
\end{table}
\end{center}

\begin{theorem}\label{thm:fibminus}
 Let $N=2n+1$, $n\geq1$. Then
\begin{equation}\label{e:fibminus}
 F_{N-i}^2 + F_i^2 = F_{N-(2j-1)} F_{N-1} + F_{2j-1},
\end{equation}
where $j=\min(i,N-i)$ and $1\leq i \leq n+1$. Observe that $N-(2j-1)$ is always even, and that the indices $N-(2j-1)$ and $2j-1$ sum to $N$. If $i=n+1$, we have Lucas's identity \eqref{eqn:lucas}.
\end{theorem}

\begin{proof}
Let us assume that $j$ is the smaller of $i$ and $N-i$, $1\leq i \leq n+1$. By Catalan's identity \eqref{eqn:catalan} and Lucas's identity \eqref{eqn:lucas},
\begin{align*}
 F_{N-i}^2 + F_i^2 &= F_{N-j}^2 + F_j^2 \\
 &= F_{(N-1)-(j-1)}^2 + F_j^2 \\
 &= F_{(N-1)-2(j-1)} F_{(N-1)} + (-1)^{(N-1)-2(j-1)}F_{j-1}^2 + F_j^2 \\
 &= F_{N-(2j-1)} F_{N-1} + F_{j-1}^2 + F_j^2 \\
 F_{N-i}^2 + F_i^2 &= F_{N-(2j-1)} F_{N-1} + F_{2j-1}.\qedhere
\end{align*}
\end{proof}

Let ${\mathcal{P}}_{-}(N)$ be the set of all pairs $[r,s]$ such that $r+s=N$ and $s$ is odd, so that $r$ is necessarily even. Note that $|{\mathcal{P}}_{-}(2n+1) |=n$. Define $\psi_{-} : {\mathcal{P}}_{-}(N) \rightarrow {\mathcal{P}}_{-}(N)$ by
$$
\psi_{-}([r,s])=[N-(2t-1),2t-1],\quad t=\min(r,s).
$$
We write $\psi_{-}$ instead of $\psi_{-}^N$ since $N$ is fixed, once chosen.  Note that $\psi_{-}([n,n+1])=[2,N-2]$ if $n$ is even, or $\psi_{-}([n+1,n])=[2,N-2]$ if $n$ is odd. Define the map $\Psi : {\mathcal{P}}_{-}(N) \rightarrow \Z|_{b}$ by $\Psi_{-}([r,s])=F_{r}.F_{s}|_b$, where $b=F_{N-1}$, so that now we require $N\geq5$. Thus, we have
\begin{equation}\label{eqn:psiminus}
 S_{b} (F_{r}.F_{s}) = \Psi_{-}(\psi_{-}([r,s])).
\end{equation}
Thus, we need only determine the cycles and fixed points of $\psi_{-}$ on ${\mathcal{P}}_{-}(N)$ to obtain the cycles and fixed points of $S_{b}$, $b=F_{N-1}$, on $\Z|_{b}$, $N\geq5$.

\begin{theorem}[$\psi_{-}$ orbit]\label{thm:psiminus}
 Let $N=2n+1$, $n\geq1$. If $n=1$ or $n=2$, then $[2,N-2]$ is a fixed point of $\psi_{-}$. If $n\geq3$, then the iterates of $\psi_{-}$ on $[2,N-2]$ comprise a cycle with initial element $[2,N-2]$ and terminal element $[n,n+1]$ if n is even, or terminal element $[n+1,n]$ if $n$ is odd.
\end{theorem}

\begin{proof}
 Observe that $(\psi_{-})^{-1}([r,2s-1])=[N-s,s]$ if $s$ is odd, and $[s,N-s]$ if $s$ is even, where $r=N-(2s-1)$. Further, since $\psi_{-}([n,n+1])=[2,N-2]$ if $n$ is even, and $\psi_{-}([n+1,n])=[2,N-2]$ if $n$ is odd, we have that $[2,N-2]$ and $[n,n+1]$, $n$ even, or $[n+1,n]$, $n$ odd, are in the same $\psi_{-}$-orbit.
\end{proof}

Let us call the $\psi_{-}$-cycle generated by $[2,N-2]$ the \emph{fundamental cycle} of $\psi_{-}$, since it occurs for every positive integer $n\geq1$. It is only meaningful for Fibonacci cycles when $n\geq2$, however. See Table \ref{tab:fibcycle-type-II} for some Fibonacci cycles, see Table \ref{tab:minuscycles} for some cycles and fixed points of $\psi_{-}$ and see Section \ref{s:fixedpoints} on fixed points.

\begin{center}
\begin{table}
\begin{tabular}{rl}\toprule
$N$ & Cycles and fixed points \\\midrule
$3$ & [[2, 1]] \\
$5$ & [[2, 3]] \\
   & [[4, 1]] \\
$7$ & [[2, 5], [4, 3]] \\
   & [[6, 1]] \\
$9$ & [[2, 7], [6, 3], [4, 5]] \\
   & [[8, 1]] \\
$11$ & [[2, 9], [8, 3], [6, 5]] \\
   & [[4, 7]] \\
   & [[10, 1]] \\
$13$ & [[2, 11], [10, 3], [8, 5], [4, 9], [6, 7]] \\
   & [[12, 1]] \\
$15$ & [[2, 13], [12, 3], [10, 5], [6, 9], [4, 11], [8, 7]] \\
   & [[14, 1]] \\
$17$ & [[2, 15], [14, 3], [12, 5], [8, 9]] \\
   & [[4, 13], [10, 7]] \\
   & [[6, 11]] \\
   & [[16, 1]] \\
$19$ & [[2, 17], [16, 3], [14, 5], [10, 9]] \\
   & [[4, 15], [12, 7], [6, 13], [8, 11]] \\
   & [[18, 1]] \\
$21$ & [[2, 19], [18, 3], [16, 5], [12, 9], [4, 17], [14, 7], [8, 13], [6, 15], [10, 11]] \\
   & [[20, 1]] \\
$23$ & [[2, 21], [20, 3], [18, 5], [14, 9], [6, 17], [12, 11]] \\
   & [[4, 19], [16, 7], [10, 13]] \\
   & [[8, 15]] \\
   & [[22, 1]] \\\bottomrule
\end{tabular}
\caption{\label{tab:minuscycles} The cycles and fixed  points of type II are obtained by the iterates of $\psi_{-}$ on $\mathcal{P}_{-}(N)$, $N=2n+1$, $n\geq1$. The iteration applies to Fibonacci cycles only when $n\geq2$, however.}
\end{table}
\end{center}

\begin{corollary}\label{cor:psiminus}
  Every element of ${\mathcal{P}}_{-}(N)$, $N=2n+1$, $n\geq1$, is either a fixed point or in a cycle of $\psi_{-}$.
\end{corollary}

\begin{proof}
 Let $z$ be an element of ${\mathcal{P}}_{-}(N)$. If $z$ is a fixed point, then we are done. Suppose $z$ is not a fixed point and consider the iterates $\{ \psi_{-}^{(k)}(z) \}$ of $z$.
\end{proof}

The following two theorems are immediate.

\begin{theorem}\label{thm:fullfibminus}
 All fixed points and cycles of $\psi_{-}$ on ${\mathcal{P}}_{-}(N)$, $N=2n+1$, $n\geq2$, generate via \eqref{e:fibminus} Fibonacci fixed points and Fibonacci cycles.
\end{theorem}

\begin{theorem}[Fibonacci cycles of type II]\label{thm:fibminuscycle}
 The iterates of $S_b$, $b=F_{2n}$, $n\geq2$, on $F_2.F_{2n-1}|_b$ comprise a cycle with initial element $F_2.F_{2n-1}|_b$ and terminal element $F_{n}.F_{n+1}|_b$ if $n$ is even, or terminal element $F_{n+1}.F_{n}|_b$ if $n$ is odd.
\end{theorem}

\section{Fixed points}\label{s:fixedpoints}

\begin{definition}[Isolated fixed point]
 A fixed point with no preimage will be called an \emph{isolated fixed point}.
\end{definition}

\begin{theorem}[Isolated fixed points of $\psi_{-}$]\label{thm:fixed-points-minus} Assume $n\geq1$.
 \begin{enumerate}
 \item The pair $[2n,4n-1]$ is an isolated fixed point of $\psi_{-}$.

 \item The pair $[2n,1]$ is an isolated fixed point of $\psi_{-}$.
\end{enumerate}
\end{theorem}

\begin{proof}
 Suppose $[2n,N-2n]$ is a fixed point of $\psi_{-}$ and $2n=\min(2n,N-2n)$. Then $[2n,N-2n]=[N-(4n-1),4n-1]$ yields $N=6n-1$ and $[2n,4n-1]$ is the fixed point.

 Suppose $[2n,N-2n]$ is a fixed point of $\psi_{-}$ and $N-2n=\min(2n,N-2n)$. Then $[2n,N-2n]=[N-(2(N-2n)-1),2(N-2n)-1]$ yields $N=2n+1$ and $[2n,1]$ is the fixed point.

 If $\psi_{-}([x,y])=[2n,4n-1]$ and $x=\min(x,y)$, then $2x-1=4n-1$ yields $x=2n$. If $\psi_{-}([x,y])=[2n,4n-1]$ and $y=\min(x,y)$, then $2y-1=4n-1$ yields $y=2n$, but $y$ is odd so this case does not occur.

 If $\psi_{-}([x,y])=[2n,1]$, then $2y-1=1$ yields $y=1$. If $x=\min(x,y)$, then $2x-1=1$ yields $x=1$, but $x$ is even so this case does not occur.
\end{proof}

The following theorem is proven similarly.

\begin{theorem}[Isolated fixed point of $\psi_{+}$]\label{thm:fixed-points-plus}
 Assume $n\geq1$. The pair $[2n,4n+1]$ is an isolated fixed point of $\psi_{+}$.
\end{theorem}

Theorems \ref{thm:fixed-points-minus} and \ref{thm:fixed-points-plus} combine to give us the following theorem.

\begin{theorem}[Fixed points in bases $F_{6n-2}$ and $F_{6n+2}$]\label{thm:fixed-points-fib} Assume $n\geq1$.

\begin{enumerate}
  \item \label{thm:fib-fixed-points-i} The number $F_{2n}.F_{4n-1}|_b$ is an isolated fixed point in base $b=F_{6n-2}$.

  \item \label{thm:fib-fixed-points-ii}The number $F_{2n}.F_{4n+1}|_b$ is an isolated fixed point in base $b=F_{6n+2}$.
\end{enumerate}
\end{theorem}

\begin{remark}
Theorem \ref{thm:fixed-points-fib} uses the following sequences in the OEIS \cite{oeis}: $F_{2n}$ \cite{A001906}, $F_{4n+1}$ \cite{A033889}, $F_{4n-1}$ \cite{A033891}, $F_{6n-2}$ \cite{A103134}, $F_{6n+2}$ \cite{A134494}.
\end{remark}

\begin{definition}[Companion base]
If $y.x|_{b}$ is fixed by $S_b$ and $x.y|_{b'}$ is also fixed by $S_{b'}$, then $b'$ is said to be a \emph{companion base} to $b$. Equivalently, $b$ and $b'$ are companion bases if and only if $n=x^2+y^2=y.x|_b=x.y|_{b'}$.
\end{definition}

\begin{theorem}[Companion bases $F_{4n-1}$ and $F_{4n}$]\label{thm:companion-fib} Assume $n\geq1$.
\begin{enumerate}
 \item \label{thm:companion-fib-a} The number $F_{2n+1}F_{2n-1}.F_{2n}F_{2n-1}|_b$ is fixed in base $b=F_{4n-1}$.

 \item \label{thm:companion-fib-b} The number $F_{2n}F_{2n-1}.F_{2n+1}F_{2n-1}|_b$ is fixed in base $b=F_{4n}$.
\end{enumerate}
\end{theorem}

\begin{remark}
Theorem \ref{thm:companion-fib} uses the following sequences in the OEIS \cite{oeis}: $F_{4n}$ \cite{A033888}, $F_{4n-1}$ \cite{A033891}, $F_{2n+1} F_{2n-1}$ \cite{A064170},  $F_{2n} F_{2n-1}$  \cite{A081016}.
\end{remark}

\begin{proof}[Proof of Theorem \ref{thm:companion-fib}] Let $y.x|_b = F_{2n+1}F_{2n-1}.F_{2n}F_{2n-1}|_b$. Thus,
\begin{align}
 x^2 + y^2 - b y - x &= F_{2n+1}^2F_{2n-1}^2 + F_{2n}^2F_{2n-1}^2 - F_{2n+1}F_{2n-1} F_{4n-1} - F_{2n}F_{2n-1} \notag\\
\intertext{Factor out and discard $F_{2n-1}$ to obtain}
 &\approx F_{2n+1}^2F_{2n-1} + F_{2n}^2F_{2n-1} - F_{2n+1} F_{4n-1} - F_{2n}  \notag\\
 &= \left( F_{2n+1}^2 + F_{2n}^2 \right) F_{2n-1}  - F_{2n+1} F_{4n-1} - F_{2n} \notag\\
 &= F_{2n-1} F_{4n+1} - F_{2n+1} F_{4n-1} - F_{2n} \quad \text{(by \eqref{eqn:lucas})} \label{egn:fib4n}\\
 \intertext{Observe that the pairs of indices are symmetric about $3n$ to obtain the identities}
 F_{2n+1}F_{4n-1} &= F_{3n}^2 + F_{n-1}^2, \label{eqn:fib3na}\\
 F_{2n-1}F_{4n+1} &= F_{3n}^2 + F_{n+1}^2. \label{eqn:fib3nb}
\intertext{Substituting \eqref{eqn:fib3na} and \eqref{eqn:fib3nb} into \eqref{egn:fib4n}, we obtain}
 &= F_{3n}^2 + F_{n+1}^2 - F_{3n}^2 - F_{n-1}^2 - F_{2n} \notag\\
 &= F_{n+1}^2 - F_{n-1}^2 - F_{2n} \notag\\
 &= F_{2n} - F_{2n} = 0\quad\text{(by \eqref{eqn:docagne})}.\notag
\end{align}
The proof for the companion base is similar and left to the reader.
\end{proof}

Here are three interesting results and a corollary on companion bases. The proof of the first is left to the reader.

\begin{theorem}[Companion bases for numbers of the form $n.n$]\label{thm:companion-odd-base}
  If $n\geq2$, then $n.n|_b$ is a fixed point in base $b=2n-1$. Thus, every odd base is it's own companion base.
\end{theorem}

\begin{theorem}[Companion bases for numbers of the form $nu.u$]\label{thm:companion-n} Let $n\geq1$, $k\geq0$, and let $u=n+1+nk$. Then the number $nu.u|_b$ is a fixed point in base $b=n^2+n+1 + (n^2+1)k$ and the number $u.nu|_{b'}$ is a fixed point in base $b'=n^3+n^2+1 +n(n^2+1)k$. Thus, $b$ and $b'$ are companion bases.
\end{theorem}

\begin{remark}
 Theorem \ref{thm:companion-n} uses the following sequences from the OEIS \cite{oeis}: $n$ \cite{A000027}, $n^2+n+1$ \cite{A002061}, $n^2+1$ \cite{A002522}, $n(n^2+1)$ \cite{A034262}, $n^3+n^2+1$ \cite{A098547}.
\end{remark}

\begin{proof}[Proof of Theorem \ref{thm:companion-n}]
 If $nu.u$ is fixed in base $b$, then
 \begin{align}
  (nu)^2 + (u)^2 &= nu \cdot b + u \notag\\
  (n^2+1)u &= n b + 1.\label{e:1}
  \intertext{Similarly, if $u.nu|_{b'}$ is fixed in base $b'$, then}
  (1+n^2)u &= b' + n.\notag
\end{align}
If we reduce equation \eqref{e:1} modulo $n$ then we obtain $u \equiv 1 \pmod n$. so that $u=1+nk$ in general. Then we obtain
\begin{align*}
 b &= \dfrac{(n^2+1)(1+nk)-1}{n} \\
   &= n + (1+n^2) k.
\end{align*}
However, this will not work since $b=n$ and $nu.u|_b$ makes no sense. However, if we take $u=n+1+nk$, then we obtain
\begin{align*}
 b &= \dfrac{(n^2+1)(n+1+nk)-1}{n} \\
   &= n^2+n+1 + (n^2+1) k,
 \intertext{and we also obtain}
 b' &= (n^2+1)(n+1+nk) - n \\
 &= n^3+n^2+1 + n(n^2+1)k.
\end{align*}
If $n=3$, for example, then it is easily verified that $12+9k.4+3k|_b$ is fixed in base $b=13+10k$ and $4+3k.12+9k|_{b'}$ is fixed in base $b'=37+30k$. Even though the proof assumed $n\geq2$, the formulas work for $n=1$ also, that is, $2+k.2+k|_b$ is fixed in base $b=3+2k$, as given by Theorem \ref{thm:companion-odd-base}.
\end{proof}

\begin{theorem}[Companion bases for numbers of the form $nu.mu$]\label{thm:companion-m-and-n} Let $m$ and $n$ be relatively prime integers such that $n>m>1$. Then the number $nu.mu|_b$ is a fixed point in base $b=b_0+m(m^2+n^2)k$, and the number $mu.nu|_{b'}$ is a fixed point in base $b'=b_0'+n(m^2+n^2)k$, where
\begin{align*}
 u &= u_0 + mnk,\\
 b_0 &= \dfrac{(m^2+n^2)u_0-m}{n},\\
 b_0' &= \dfrac{(m^2+n^2)u_0-n}{m},
\end{align*}
and where, according to the Chinese Remainder Theorem, $u_0$ the smallest solution of the congruences
$$
u \equiv n^{-1} \pmod m\quad\text{and}\quad u \equiv m^{-1} \pmod n.
$$
Thus, $b$ and $b'$ are companion bases.
\end{theorem}

\begin{proof}
If $nu.mu|_b$ is a fixed point in base $b$, then
\begin{align}
 (mu)^2 + (nu)^2 &= nu\cdot b + mu \notag\\
 (m^2+n^2) u &= nb + m\label{e:m}\\
 \intertext{Similary, If $mu.nu|_{b'}$ is a fixed point in base $b'$, then}
 (m^2+n^2) u &= mb' + n. \label{e:n}
\end{align}
If we reduce \eqref{e:m} modulo $n$ and reduce \eqref{e:n} modulo $m$, we obtain the congruences
$$
u \equiv m^{-1} \pmod n\quad\text{and}\quad u \equiv n^{-1} \pmod m.
$$
Let $u_0$ be the smallest positive solution to this pair of congruences guaranteed by the Chinese Remainder Theorem, so that $u = u_0 + mnk$, $k\geq0$. The bases $b$ and $b'$ are then given by the arithmetic sequences
\begin{align*}
 b &= \dfrac{(m^2+n^2)(u_0+mnk)-m}{n} \\
 &= \dfrac{(m^2+n^2)u_0-m}{n} + m(m^2+n^2)k,
 \intertext{and}
 b' &= \dfrac{(m^2+n^2)u_0-n}{m} + n(m^2+n^2)k. \qedhere
\end{align*}
\end{proof}

An immediate application of Theorem \ref{thm:companion-m-and-n} is the following.

\begin{corollary}[Companion bases for $(n+1)u.nu$]\label{cor:companion-consecutive}  Let $n\geq1$ and $k\geq0$. Then  $(n+1)u.nu|_b$ is fixed in base $b$ and $nu.(n+1)u|_{b'}$ is fixed in base $b'$, where
\begin{align*}
 u &= 2n+1 + n(n+1)k, \\
 b &= 4n^2 + 2n + 1 + n(2n^2 + 2n + 1)k, \\
 b' &= 4n^2 + 6n + 3 + (n+1)(2n^2 + 2n + 1)k.
\end{align*}
\end{corollary}

\begin{remark}
 Corollary \ref{cor:companion-consecutive} uses the following sequences from the OEIS \cite{oeis}: $2n+1$ \cite{A005408}, $n(n+1)$ \cite{A002378}, $n(2n^2+2n+1)$ \cite{A048395}, $4n^2+6n+3$ \cite{A054554}, $4n^2+2n+1$ \cite{A054569}, $(n+1)(2n^2+2n+1)$ \cite{A059722}.
\end{remark}

Since $\gcd(F_{m},F_{n})=F_{\gcd(m,n)}$ in general, Theorem \ref{thm:companion-fib} has an extension to arithmetic sequences of companion bases.

\begin{corollary}[Arithmetic sequence of Fibonacci companion bases]\label{thm:companion-arith-fib}
 Assume $n\geq1$ and $k\geq0$. The number $F_{2n+1}u.F_{2n}u|_{b}$ is a fixed point in base $b$, and $F_{2n}u.F_{2n+1}u|_{b'}$ is fixed in base $b'$, where
\begin{align*}
 u &= F_{2n-1}+F_{2n}F_{2n+1}k, \\
 b &= F_{4n-1}+F_{2n}F_{4n+1}k, \\
 b' &= F_{4n}+F_{2n+1}F_{4n+1}k.
\end{align*}
\end{corollary}

\begin{remark}
Theorem \ref{thm:companion-arith-fib} uses the following sequences in the OEIS \cite{oeis}: $F_{2n-1}$ \cite{A001519}, $F_{2n}$ \cite{A001906}, $F_{4n}$ \cite{A033888}, $F_{4n+1}$ \cite{A033889}, $F_{4n-1}$ \cite{A033891}, $F_{2n}F_{2n+1}$ \cite{A081018}. The sequences $F_{2n}F_{4n+1}$ and $F_{2n+1}F_{4n+1}$ do not occur in the OEIS.
\end{remark}

Let us provide one more example of natural interest.

\begin{theorem}[Companion bases for triangular numbers]\label{thm:companion-triangular}  Let $T_n=n(n+1)/2$, $n\geq2$, be a triangular number. Then $T_n.T_{n+1}|_b$ is fixed in base $b=n^2+n+1$ and $T_{n+1}.T_n|_{b'}$ is fixed in base $b'=(n+1)^2+(n+1)+1=n^2+3n+3$.
\end{theorem}

\begin{remark}
Theorem \ref{thm:companion-triangular} uses the following sequences from the OEIS \cite{oeis}: $T_n$ \cite{A000217}, $n^2+n+1$ \cite{A002061}, $T_{n^2}$ \cite{A037270}. Recall that $T_n^2 + T_{n+1}^2 = T_{(n+1)^2}$.
\end{remark}

\begin{proof}[Proof of Theorem \ref{thm:companion-triangular}] The proof is left to the reader. \end{proof}

If we have an arithmetic sequence of Fibonacci fixed points as in Theorem  \ref{thm:companion-arith-fib}, then it should come as no surprise that we also have arithmetic sequences of Fibonacci cycles, and that's the topic of the next two sections.

\section{Arithmetic progressions of Fibonacci cycles of type I}\label{s:arithFibonacci-i}

This section shows that the fundamental cycles of type I admit extension to an arithmetic sequence of cycles in which the common differences are all Fibonacci numbers.

\begin{theorem}\label{thm:arithfibcycle-i}
 Let $N=2n-1$, $n\geq2$, and $k\geq0$. Then
\begin{equation}
  S_b(x_1.x_0) = y_1.y_0|_b,
 \end{equation}
where
\begin{align*}
 b &= F_{N+1} + F_{N+2}k, \\
 x_0 &= F_{i} + F_{i+1}k, \\
 x_1 &= F_{N-i} + F_{N-i+1}k, \\
 y_0 &= F_{2i+1} + F_{2i+2}k, \\
 y_1 &= F_{N-(2i+1)} + F_{N-2i}k,
\end{align*}
and where $0\leq i\leq n-1$.
\end{theorem}

We need a couple of lemmas to verify Theorem \ref{thm:arithfibcycle-i}.

\begin{lemma}\label{lem:dOcagne}
Let $N=2n-1$, $n\geq2$. Then
\begin{equation}\label{eqn:odd-docagne}
  F_{N-2i} F_{N+1} = F_{N+2}F_{N-2i-1} + F_{2i+2},
\end{equation}
where $0\leq i \leq n-1$.
\end{lemma}

\begin{proof}
Let $m=N+1$ and $n=N-2i-1$ in d'Ocagne's identity \eqref{eqn:docagne} to obtain
\begin{align*}
 F_{N+1} F_{(N-2i-1)+1} - F_{N+2}F_{N-2i-1} &= (-1)^{N-2i-1}F_{N+1-N+2i+1}\\
 F_{N+1} F_{N-2i} - F_{N+2}F_{N-2i-1} &= F_{2i+2}\\
 F_{N+1} F_{N-2i} &= F_{N+2}F_{N-2i-1} + F_{2i+2}.\qedhere
\end{align*}
\end{proof}

\begin{lemma}\label{lem:odd-vadja}
Let $N=2n-1$, $n\geq2$. Then
\begin{equation}\label{eqn:odd-vadja}
  F_{N-2i}F_{N+1} = F_{N-i}F_{N-i+1} + F_{i} F_{i+1}
\end{equation}
where $0\leq i \leq n-1$.
\end{lemma}

\begin{proof}
Observe that $N+1-N+2i=2i+1$ so we let $n=N-2i$, $r=i$, $s=i+1$ in Vadja's identity \eqref{eqn:vajda} to obtain
\begin{align*}
 F_{N-2i}F_{N+1} &= F_{N-2i+i}F_{N-2i+i+1} - (-1)^{N-2i}F_{i} F_{i+1}\\
 &= F_{N-i}F_{N-i+1} + F_{i} F_{i+1}.\qedhere
\end{align*}
\end{proof}

\begin{proof}[Proof of Theorem \ref{thm:arithfibcycle-i}]
 Let $N=2n-1$, $n\geq2$. Consider $F_N$ in base $F_{N+1}$. Thus,
\begin{align}
 x_0^2 &+ x_1^2 - y_1b - y_0 \notag\\
 &= (F_{i} + F_{i+1}k)^2 + (F_{N-i} + F_{N-i+1}k)^2 - (F_{N-2i-1} + F_{N-2i}k)(F_{N+1} + F_{N+2}k)\notag\\
 &\quad- (F_{2i+1} + F_{2i+2}k) \notag\\
 &=(F_{i}^2 + F_{N-i}^2 - F_{N-2i-1}F_{N+1} - F_{2i+1}) \label{eqn:term1}\\
 &\qquad+(2F_{i}F_{i+1} + 2F_{N-i}F_{N-i+1} - F_{N-2i-1}F_{N+2} - F_{N-2i}F_{N+1} - F_{2i+2})k \label{eqn:term2}\\
 &\qquad+(F_{i+1}^2 + F_{N-i+1}^2 - F_{N-2i}F_{N+2})k^2 \label{eqn:term3}
\end{align}
Let's take each term separately. The first term \eqref{eqn:term1} is just Theorem \ref{thm:fibplus} and has already been verified.

Let us consider the $k$-term, \eqref{eqn:term2}. Thus,
\begin{align*}
 &2F_{i}F_{i+1} + 2F_{N-i}F_{N-i+1} - F_{N-2i-1}F_{N+2} - F_{N-2i}F_{N+1} - F_{2i+2} \\
 &=2\left(F_{i}F_{i+1} + F_{N-i}F_{N-i+1}\right) - F_{N-2i-1}F_{N+2} - F_{N-2i}F_{N+1} - F_{2i+2} \\
 &=2F_{N-2i}F_{N+1} - F_{N-2i-1}F_{N+2} - F_{N-2i}F_{N+1} - F_{2i+2} \quad \text{(by \eqref{eqn:odd-vadja})}\\
 &=F_{N-2i}F_{N+1} - F_{N-2i-1}F_{N+2} - F_{2i+2} \\
 &= 0.\quad \text{(by \eqref{eqn:odd-docagne})}
\end{align*}

Let us consider the $k^2$-term, \eqref{eqn:term3}. Using Catalan's identity \eqref{eqn:catalan} with $r=i+1$ we obtain
\begin{align*}
 F_{i+1}^2 + F_{N-i+1}^2 - F_{N-2i}F_{N+2} &= F_{i+1}^2 + F_{N-2i} F_{N+2} + (-1)^{N-2i}F_{i+1}^2 - F_{N-2i}F_{N+2} \\
 &= F_{i+1}^2 + F_{N-2i} F_{N+2} - F_{i+1}^2 - F_{N-2i}F_{N+2} \\
 &= 0.\qedhere
\end{align*}
\end{proof}

Thus, iterates of a sum of squares map can be replaced by iterates on pairs of nonnegative integers. Consider pairs $[[r,r+1],[s,s+1]]$, where $r+y=N$, $N=2n-1$, $n\geq2$, and extend $\psi_{+}$ defined in Section \ref{s:fibonacci-ii} by $\psi_{+}([[r,r+1],[s,s+1]])=[[N-(2t+1),N-2t],[2t+1,2t+2]]$, where $t=\min(r,s)$. The following theorem is proven similarly to Theorem \ref{thm:psiplus}.

\begin{theorem}[Arithmetic $\psi_{+}$]\label{thm:arith-psiplus}
  Let $N=2n-1$, $n\geq1$. If $n=1$, then $[[0,1],[1,2]]$ is a fixed point of $\psi_{+}$. If $n\geq2$, then the iterates of $\psi_{+}$ on $[[0,1],[N,N+1]]$ comprise a cycle with initial element $[[0,1],[N,N+1]]$ and terminal element $[[n,n+1],[n-1,n]]$ if $n$ is even, or terminal element $[[n-1,n],[n,n+1]]$ if $n$ is odd.
\end{theorem}

Assume $n\geq2$ and let $N=2n-1$. Define
$$
\Psi_{+}([[r,r+1],[s,s+1]])=F_{r}+F_{r+1}k.F_{s}+F_{s+1}k|_{b},\quad\text{where $b=F_{N+1}+F_{N+2}k$,}
$$
and $k$ is any nonnegative integer. By Theorem \ref{thm:arithfibcycle-i}, we have
$$
 S_b( F_{r}+F_{r+1}k.F_{s}+F_{s+1}k|_{b} ) = \Psi_{+}(\psi_{+}([[r,r+1],[s,s+1]])), \quad (r+s=N).
$$
Consequently, we have the following theorem.

\begin{theorem}[Arithmetic Fibonacci cycles of type I] \label{thm:arithFibonacci-i}
Let $N=2n-1$, $n\geq2$, and $k\geq0$. Then there is a cycle of $S_b$, $b=F_{N+1}+F_{N+2}k$, with initial element $F_{0}+F_{1}k.F_{N}+F_{N+1}k|_b$ and terminal element $F_{n}+F_{n+1}k.F_{n-1}+F_{n}k|_b$ if $n$ is even, or terminal element $F_{n-1}+F_{n}k.F_{n}+F_{n+1}k|_b$ if $n$ is odd.
\end{theorem}

\section{Arithmetic progressions of Fibonacci cycles of type II}\label{s:arithFibonacci-ii}

This section shows that the fundamental cycles of type II admit extension to an arithmetic sequence of cycles in which the common differences are all Fibonacci numbers.
\begin{theorem}\label{thm:arithfibcycle-ii}
 Let $N=2n+1$, $n\geq1$, and $k\geq0$. Then
\begin{equation}
  S_b(x_1.x_0) = y_1.y_0|_b,
 \end{equation}
where
\begin{align*}
 b &= F_{N-1} + F_{N-2}k, \\
 x_0 &= F_{i} + F_{i-1}k, \\
 x_1 &= F_{N-i} + F_{N-i-1}k, \\
 y_0 &= F_{2i-1} + F_{2i-2}k, \\
 y_1 &= F_{N-(2i-1)} + F_{N-2i}k,
\end{align*}
and where $1\leq i\leq n$.
\end{theorem}

\begin{proof} The proof is similar to that of Theorem \ref{thm:arithfibcycle-i}, using Lemmas \ref{e:lemma3} and \ref{e:lemma4}. \end{proof}

\begin{lemma}
\begin{equation}\label{e:lemma3}
  F_i F_{i-1} + F_{N-i} F_{N-i-1} = F_{N-2i} F_{N-1}.
\end{equation}
\end{lemma}

\begin{lemma}
\begin{equation}\label{e:lemma4}
  F_{N-2i} F_{N-1} = F_{N-(2i-1)} F_{N-2} + F_{2i-2}
\end{equation}
\end{lemma}

Thus, iterates of a sum of squares map can be replaced by iterates on pairs of nonnegative integers. Consider pairs $[[r,r-1],[s,s-1]]$, where $r+s=N$, $N=2n+1$, $n\geq1$, and extend $\psi_{-}$ defined in Section \ref{s:fibonacci-i} by $\psi_{-}([[r,r-1],[s,s-1]])=[[N-(2t-1),N-2t],[2t-1,2t-2]]$, where $t=\min(r,s)$. The following theorem is proven similarly to Theorem \ref{thm:psiminus}.

\begin{theorem}[Arithmetic $\psi_{-}$]\label{thm:arith-psiminus}
  Let $N=2n+1$, $n\geq1$. The iterates of $\psi_{-}$ on $[[2,1],[N-2,N-3]]$ comprise a cycle with initial element $[[2,1],[N-2,N-3]]$  and with terminal element $[[n,n-1],[n+1,n]]$ if $n$ is even, or terminal element $[[n+1,n],[n,n-1]]$ if $n$ is odd.
\end{theorem}

Define
$$
\Psi_{-}([[r,r-1],[s,s-1]])=F_{r}+F_{r-1}k.F_{s}+F_{s-1}k|_{b},
$$
where $b=F_{2n}+F_{2n-1}k$, $n\geq2$, and $k\geq0$. By Theorem \ref{thm:arithfibcycle-i}, we have
$$
 S_b( F_{r}+F_{r-1}k.F_{s}+F_{s-1}k|_{b} ) = \Psi_{-}(\psi_{-}([[r,r-1],[s,s-1]])),\quad (r+s=N).
$$
Consequently, we have the following theorem.

\begin{theorem}[Arithmetic Fibonacci cycles of type II] \label{thm:arithFibonacci-ii}
Let $N=2n+1$, $n\geq2$, and $k\geq0$. Then there is a cycle of $S_b$, $b=F_{2n}+F_{2n-1}k$, with initial element $F_{2}+F_{1}k.F_{2n-1}+F_{2n-2}k|_b$ and terminal element $F_{n}+F_{n-1}k.F_{n+1}+F_{n}k|_b$ if $n$ is even, or terminal element $F_{n+1}+F_{n}k.F_{n}+F_{n-1}k|_b$ if $n$ is odd.
\end{theorem}

\section{Generalizations to Pell polynomials}\label{s:pell}

\begin{definition}[Pell polynomials, \cite{koshy2014}]
  The Pell polynomials are defined recursively by
\begin{align}
 p_0(x) &= 0, \quad p_1(x) = 1, \\
 p_n(x) &= 2x p_{n-1}(x) + p_{n-2}(x), \quad n\geq2,
\end{align}
Note that $p_2(x)=2x$. See Theorem \ref{thm:pellminus}.
\end{definition}

\begin{remark}
The following results apply equally well to the Fibonacci polynomials defined by $f_n(x)=p_n(x/2)$, $n\geq1$ \cite{koshy2014}.
\end{remark}

\begin{theorem}[Pell polynomial identities, \cite{koshy2014}]\label{thm:pell-identities}
The Pell polynomials satisfy the following identities:
\begin{description}
  \item[Pell-Cassini identity]
  \begin{equation}\label{eqn:pell-cassini}
    p_{n}^2(x) = p_{n+1}(x) p_{n-1}(x) + (-1)^{n+1}
  \end{equation}
  \item[Pell-Catalan identity]
  \begin{equation}\label{eqn:pell-catalan}
    p_{n}^2(x) = p_{n+r}(x)p_{n-r}(x) + (-1)^{n-r} p_r^2(x)
  \end{equation}
  \item[Pell-Vajda identity]
  \begin{equation}\label{eqn:pell-vajda}
    p_{n+r}(x) p_{n+s}(x) = p_{n}(x)p_{n+r+s}(x) + (-1)^n p_r(x) p_s(x)
  \end{equation}
  \item[Pell-Lucas identity]
  \begin{equation}\label{eqn:pell-lucas}
    p_{2n+1}(x) = p_{n+1}^2(x) + p_{n}^2(x)
  \end{equation}
  \item[Pell-d'Ocagne identity]
  \begin{equation}\label{eqn:pell-docagne}
    2x\,p_{2n}(x) = p_{n+1}^2(x) - p_{n-1}^2(x)
  \end{equation}
\end{description}
\end{theorem}

\begin{remark}
The ``Pell'' prefix is intended as descriptive and is unrelated to the historical discovery of the identities.
\end{remark}

By Theorem \ref{thm:pell-identities}, the same properties discovered in Sections \ref{s:fibonacci-i} and \ref{s:fibonacci-ii} apply to the Pell polynomials. Consequently, we only summarize the results.

\begin{theorem}[Pell identity of type I]\label{thm:pell-type-I}
 Let $N=2n-1$, $n\geq1$. Then
 \begin{equation}\label{e:pell-type-II}
    p_{N-i}^2(x) + p_i^2(x) = p_{N-(2j+1)}(x) p_{N+1}(x) + p_{2j+1}(x), \\
 \end{equation}
where $j=\min(i,N-i)$ and $0\leq i \leq n-1$. Observe that $N-(2j+1)$ is always even, and that the indices $N-(2j+1)$ and $2j+1$ sum to $N$.
\end{theorem}

\begin{theorem}[Pell cycles of type I]\label{thm:pellplus}
 The iterates of $S_b$, $b=p_{2n}(x)$, yield a cycle with initial element $p_0(x).p_{2n-1}(x)|_b$ and terminal element $p_{n}(x).p_{n-1}(x)|_b$ if $n$ is even, or terminal element $p_{n-1}(x).p_{n}(x)|_b$ if $n$ is odd.
\end{theorem}

\begin{theorem}[Pell identity of type II]\label{thm:pell-type-II}
 Let $N=2n+1$, $n\geq1$. Then
 \begin{equation}\label{e:pell-type-I}
    p_{N-i}^2(x) + p_i^2(x) = p_{N-(2j-1)}(x) p_{N-1}(x) + p_{2j-1}(x), \\
 \end{equation}
where $j=\min(i,N-i)$ and $1\leq i \leq n+1$. Observe that $N-(2j-1)$ is always even, and that the indices $N-(2j-1)$ and $2j-1$ sum to $N$.
\end{theorem}

\begin{theorem}[Pell cycles of type II]\label{thm:pellminus}
 The iterates of $S_b$, $b=p_{2n}(x)$, yield via \eqref{e:pell-type-I} a cycle with initial element $p_2(x).p_{2n-1}(x)|_b$ and terminal element $p_{n}(x).p_{n+1}(x)|_b$ if $n$ is even, or terminal element $p_{n+1}(x).p_{n}(x)|_b$ if $n$ is odd.
\end{theorem}

\begin{theorem}[Pell fixed points]\label{thm:pell-fixed-points}
 Let $n$ be a positive integer.
\begin{enumerate}
\renewcommand{\labelenumi}{\theenumi}
\renewcommand{\theenumi}{(\alph{enumi})}

  \item\label{thm:pell-fixed-points-a} The polynomial $p_{2n}(x).p_{4n-1}(x)|_b$ is a fixed point of $S_b$, where $b=p_{6n-2}(x)$.
  \item\label{thm:pell-fixed-points-b} The polynomial $p_{2n}(x).p_{4n+1}(x)|_b$ is a fixed point of $S_b$, where $b=p_{6n+2}(x)$.
  \item\label{thm:pell-fixed-points-c} The polynomial $p_{2n}(x)p_{2n-1}(x).p_{2n+1}(x)p_{2n-1}(x)|_b$ is a fixed point of $S_b$, where $b=p_{4n}(x)$.

\end{enumerate}
\end{theorem}

Theorem \ref{thm:pell-fixed-points} \ref{thm:pell-fixed-points-c} generalizes Theorem \ref{thm:companion-fib} \ref{thm:companion-fib-b}. Unfortunately, Theorem \ref{thm:companion-fib} \ref{thm:companion-fib-a} does not generalize to Pell polynomials because of the Pell-d'Ocagne identity, \eqref{eqn:pell-docagne}.

\begin{proof}[Proof of \ref{thm:pell-fixed-points} \ref{thm:pell-fixed-points-c}] Thus,
\begin{align*}
 \left(p_{2n}(x)p_{2n-1}(x)\right)^2 &+ \left(p_{2n+1}(x)p_{2n-1}(x)\right)^2 - p_{2n}(x)p_{2n-1}(x)p_{4n}(x) - p_{2n+1}(x)p_{2n-1}(x) \\
 &\quad\text{\footnotesize(remove the factor $p_{2n-1}(x)$)} \\
 &\approx p_{2n}(x)^2 p_{2n-1}(x) + p_{2n+1}(x)^2 p_{2n-1}(x) - p_{2n}(x)p_{4n}(x) - p_{2n+1}(x)\\
 &= \left( p_{2n}(x)^2 + p_{2n+1}(x)^2 \right) p_{2n-1}(x) - p_{2n}(x)p_{4n}(x) - p_{2n+1}(x)\\
 &= p_{4n+1}(x) p_{2n-1}(x) - p_{2n}(x)p_{4n}(x) - p_{2n+1}(x)\\
 &= p_{3n}(x)^2+p_{n+1}(x)^2 - \left( p_{3n}(x)^2 - p_{n}(x)^2 \right) - p_{2n+1}(x)\\
 &= p_{n+1}(x)^2 + p_{n}(x)^2 - p_{2n+1}(x)\\
 &= 0.\quad\text{(by \eqref{eqn:pell-lucas})}\qedhere
\end{align*}
\end{proof}

Theorem \ref{thm:pell-fixed-points} \ref{thm:pell-fixed-points-c} has an extension to arithmetic sequences of Pell polynomials.

\begin{corollary}[Arithmetic sequence of fixed points in Pell polynomials]\label{cor:arithpellfixed}
 Assume $n\geq1$ and $k\geq0$. The polynomial $p_{2n}(x)u.p_{2n+1}(x)u|_{b}$ is a fixed point in base $b=p_{4n}(x)+p_{2n+1}(x)p_{4n+1}(x)k$, where $u=p_{2n-1}(x)+p_{2n}(x)p_{2n+1}(x)k$.
\end{corollary}

\begin{proof} The proof is similar to that of Theorem \ref{thm:pell-fixed-points} \ref{thm:pell-fixed-points-c}. \end{proof}

\begin{theorem}\label{thm:pellfibcycle-i}
 Let $N=2n-1$, $n\geq1$, and let $k\geq0$. Then
\begin{equation}\label{e:pellfibcycle-i}
  S_b(x_1.x_0) = y_1.y_0,
 \end{equation}
where
\begin{align*}
 b &= p_{N+1}(x) + p_{N+2}(x)k, \\
 x_0 &= p_{i}(x) + p_{i+1}(x)k, \\
 x_1 &= p_{N-i}(x) + p_{N-i+1}(x)k, \\
 y_0 &= p_{2i+1}(x) + p_{2i+2}(x)k, \\
 y_1 &= p_{N-(2i+1)}(x) + p_{N-2i}(x)k.
\end{align*}
and where $0\leq i\leq n-1$.
\end{theorem}

\begin{proof}
 The proof is similar to that of Theorem \ref{thm:arithfibcycle-i}.
\end{proof}

A direct application of Theorem \ref{thm:pellfibcycle-i} yields a generalization of Theorem \ref{thm:fibpluscycle} to Pell polynomials.

\begin{theorem}[Arithmetic Pell cycles of type I]\label{thm:arithpellplus}
 Let $N=2n-1$, $n\geq 1$, and $k\geq0$. Then the iterates of $S_b$, $b=p_{N+1}(x)+p_{N+2}(x)k$, yield via \eqref{e:pellfibcycle-i} a cycle with initial element
 $$
 p_{0}(x)+p_{1}(x)k.p_{N}(x)+p_{N+1}(x)k|_b
 $$
 and terminal element
 $$
 p_{n}(x)+p_{n+1}(x)k.p_{n-1}(x)+p_{n}(x)k|_b
 $$
 if $n$ is even, or terminal element
 $$
 p_{n-1}(x)+p_{n}(x)k.p_{n}(x)+p_{n+1}(x)k|_b
 $$
 if $n$ is odd.
\end{theorem}

\begin{proof} The proof is left to the reader. \end{proof}

\begin{theorem}\label{thm:pellfibcycle-ii}
Let $N=2n+1$, $n\geq1$, and let $k\geq0$. Then
\begin{equation}\label{e:pellfibcycle-ii}
  S_b(x_1.x_0) = y_1.y_0,
 \end{equation}
where
\begin{align*}
 b &= p_{N-1}(x) + p_{N-2}(x)k, \\
 x_0 &= p_{i}(x) + p_{i-1}(x)k, \\
 x_1 &= p_{N-i}(x) + p_{N-i-1}(x)k, \\
 y_0 &= p_{2i-1}(x) + p_{2i-2}(x)k, \\
 y_1 &= p_{N-(2i-1)}(x) + p_{N-2i}(x)k,
\end{align*}
and where $1\leq i\leq n+1$.
\end{theorem}

\begin{proof}
 The proof is similar to that of Theorem \ref{thm:arithfibcycle-ii}.
\end{proof}

A direct application of Theorem \ref{thm:pellfibcycle-ii} yields a generalization of Theorem \ref{thm:fibminuscycle} to Pell polynomials.

\begin{theorem}[Arithmetic Pell cycles of type II]\label{thm:arithpellminus}
 Let $N=2n+1$, $n\geq 1$, and $k\geq0$. Then the iterates of $S_b$, $b=p_{N-1}(x)+p_{N-2}(x)k$, yield via \eqref{e:pellfibcycle-ii} a cycle with initial element
 $$
 p_{2}(x)+p_{1}(x)k.p_{N-2}(x)+p_{N-3}(x)k|_b
 $$
 and terminal element
 $$
 p_{n}(x)+p_{n-1}(x)k.p_{n+1}(x)+p_{n}(x)k|_b
 $$
 if $n$ is even, or terminal element
 $$
 p_{n+1}(x)+p_{n}(x)k.p_{n}(x)+p_{n-1}(x)k|_b
 $$
 if $n$ is odd.
\end{theorem}

\begin{proof}
 The proof is similar to that of Theorem \ref{thm:arithfibcycle-ii}.
\end{proof}

\medskip

\noindent MSC2020: 11B25, 11B39

\end{document}